\documentclass[preprint,10pt]{elsarticle}

\usepackage{amsmath}
\usepackage{amssymb}
\usepackage{amsthm}

\theoremstyle{plain}
\newtheorem{theorem}{Theorem}
\newtheorem*{problem*}{Problem}

\theoremstyle{definition}

\theoremstyle{remark}
\newtheorem*{remark*}{Remark}

\journal{TBD}

\begin{document}

\begin{frontmatter}
  \title{A lower bound on the size of an absorbing set in an arc-coloured tournament}
  \author[uca]{L.~Beaudou\corref{cor1}}
  \ead{laurent.beaudou@uca.fr}
  \author[mcgill]{L.~Devroye}
  \ead{lucdevroye@gmail.com}
  \author[udm]{G.~Hahn}
  \ead{hahn@iro.umontreal.ca}
  
  \cortext[cor1]{Corresponding author}
  \address[uca]{Universit\'e Clermont Auvergne, Clermont-Ferrand, France}
  \address[mcgill]{McGill University, Montr\'eal, Qu\'ebec, Canada}
  \address[udm]{Universit\'e de Montr\'eal, Montr\'eal, Qu\'ebec, Canada}
  \begin{abstract}
    Bousquet, Lochet and Thomass\'e recently gave an elegant proof
    that for any integer $n$, there is a least integer $f(n)$ such
    that any tournament whose arcs are coloured with $n$ colours
    contains a subset of vertices $S$ of size $f(n)$ with the property
    that any vertex not in $S$ admits a monochromatic path to some
    vertex of $S$. In this note we provide a lower bound on the value
    $f(n)$.
  \end{abstract}
 
\end{frontmatter}

A {\em directed graph} or {\em digraph} $D$ is a pair $(V,A)$ where
$V$ is a set called the {\em vertex set} of $D$ and $A$ is a subset of $V^2$
called the {\em arc set} of $D$. When $(u,v)$ is in $A$, we say there is an
arc from $u$ to $v$. A {\em tournament} is a digraph for which there is
exactly one arc between each pair of vertices (in only one direction). 

Given a colouring of the arcs of a digraph, we say that a vertex $x$
{\em is absorbed} by a vertex $y$ if there is a monochromatic path from $x$
to $y$. A subset $S$ of $V$ is called an {\em absorbing set} if any
vertex in $V \setminus S$ is absorbed by some vertex in $S$.

In the early eighties, Sands, Sauer and Woodrow~\cite{ssw1982}
suggested the following problem (also attributed to Erd\H{o}s in the
same paper): 

\begin{problem*}[Sands, Sauer and Woodrow \cite{ssw1982}]
  For each $n$, is there a (least) positive integer $f(n)$ so that
  every finite tournament whose arcs are coloured with $n$ colours contains an
  absorbing set $S$ of size $f(n)$ ?
\end{problem*}

This problem has been investigated in weaker settings by forbidding
some structures (see \cite{m1988}, \cite{gr2004}), or by making
stronger claims on the nature of the tournament (see
\cite{pg2014}). Hahn, Ille and Woodrow~\cite{hiw2004} also approached
the infinite case.

Recently, P\'alv\H{o}lgyi and Gy\'arf\'as~\cite{pg2014} have shown
that a positive answer to this problem would imply a new proof of a
former result from B\'ar\'any and Lehel~\cite{bl1987} stating that any
set $X$ of points in $\mathbb{R}^d$ can be covered by $f(d)$ $X$-boxes
(each box is defined by two points in $X$). In 2017, Bousquet Lochet
and Thomass\'e~\cite{blt2017} gave a positive answer to the problem of
Sands, Sauer and Woodrow.

\begin{theorem}[Bousquet, Lochet and Thomass\'e \cite{blt2017}]
Function $f$ is well defined and $f(n) = O(\ln(n) \cdot n^{n+2})$.
\end{theorem}

In this note, we provide a lower bound on the value of $f(n)$.

\begin{theorem}
  For any integer $n$, let 
  \begin{equation*}
     p =  {n-1 \choose \left\lfloor \frac{n-1}{2} \right\rfloor}.
  \end{equation*} There is a tournament arc-coloured with $n$
  colours, such that no set of size less than $p$ absorbs the rest of
  the tournament, and thus $f(n) \geq p$.
  \label{thm:main}
\end{theorem}

\begin{proof}
  Let $n$ be an integer and let $\mathcal{P}$ be the family of all
  subsets of $\{1,\ldots,n-1\}$ of size exactly $\left\lfloor
  \frac{n-1}{2} \right\rfloor$. Note that $\mathcal{P}$ has size $p$.

  For any integer $m$, let $V(m)$ be the set $\{1,\ldots,m\} \times
  \mathcal{P}$ and let $\mathcal{T}(m)$ be the probability space
  consisting of arc-coloured tournaments on $V(m)$ where the
  orientation of each arc is fairly determined (probability
  $1/2$ for each orientation), and the colour of an arc from
  $(i,P)$ to $(j,P')$ is $n$ if $P=P'$, and picked randomly in $P'
  \setminus P$ otherwise. Note that if $P$ and $P'$ are distinct,
  then $P' \setminus P$ can not be empty since both sets have same
  size.

  For any set $P$ in $\mathcal{P}$, we shall denote by $B(P)$ the set
  of vertices having $P$ as second coordinate. We may call this set,
  the {\em bag of} $P$. A key observation is that the arcs coming in a
  bag, the arcs leaving a bag and the arcs contained in a bag share no
  common colour. As a consequence, a monochromatic path is either
  contained in a bag or of length 1.
  
  Let us find an upper bound on the probability that such a tournament
  is absorbed by a set of size strictly less than $p$. Let $S$ be a
  subset of $V(m)$ of size  $p-1$. There must exist a set $P$
  in $\mathcal{P}$ such that $S$ does not hit $B(P)$. For $S$ to be
  absorbing, each vertex in $B(P)$ must have an outgoing arc to some
  vertex of $S$. Let $x$ be a vertex in $B(P)$,
  \begin{equation*}
    Pr(x \text{ is absorbed by } S) = 1 - \left( \frac{1}{2} \right)^{p-1}.
  \end{equation*}
  The events of being absorbed by $S$ are pairwise independent for
  elements of $B(P)$. Then,
  \begin{equation*}
    Pr(B(P) \text{ is absorbed by } S) = \left(1 - \left( \frac{1}{2} \right)^{p-1}\right)^m.
  \end{equation*}
  This is an upper bound on the probability for $S$ to absorb the
  whole tournament. We may sum this for every possible choice of $S$.
  \begin{equation*}
    Pr(\text{some } S \text{ is absorbing}) \leq {mp \choose p-1}  \left(1 - \left( \frac{1}{2} \right)^{p-1}\right)^m.
  \end{equation*}

  Finally, by using the classic inequalities ${n \choose k} \leq
  \left(\frac{en}{k}\right)^k $ and $(1+x)^n \leq e^{nx}$, we obtain
  \begin{equation*}
    Pr(\text{some } S \text{ is absorbing}) \leq \left(\frac{emp}{p-1}\right)^{p-1} e^{-m\left(\frac{1}{2}\right)^{p-1}}.
  \end{equation*}
  When $m$ tends to infinity, this last quantity tends to 0. So that,
  for $m$ large enough, there is a tournament which is not absorbed by
  $p-1$ vertices.  
\end{proof}

\begin{remark*}
By Stirling's approximation we derive that the bound obtained in
Theorem~\ref{thm:main} is of order $\frac{2^n}{\sqrt{n}}$.
\end{remark*}

\section*{Acknowledgements}

The authors are thankful to V. Chv\'atal and the (late) ConCoCO
seminar for allowing this research to happen.

\end{document}